\documentclass[12pt,a4paper,reqno]{amsart}
\usepackage{amssymb, amsmath}
\usepackage{geometry,enumerate,stackrel,mathtools}
\usepackage{enumitem,centernot,color}
\usepackage[colorlinks=true,linkcolor=blue,urlcolor=blue,citecolor=blue]{hyperref}
\newtheorem{theorem}{Theorem}[section]

\newtheorem{lemma}[theorem]{Lemma}
\newtheorem{corollary}[theorem]{Corollary}
\newtheorem{proposition}[theorem]{Proposition}

\theoremstyle{definition}

\newtheorem{remark}[theorem]{Remark}

\begin{document}

\title[Dual Hardy operator minus identity]{Best bounds for the dual Hardy operator minus identity on decreasing functions}

\author[A. Ben Said]{Achraf Ben Said$^{*}$}
\author[G. Sinnamon]{Gord Sinnamon$^{**}$}

\address{Department of Analysis and Applied Mathematics, Complutense University  of Madrid, 28040 Madrid, Spain.}
\email{achbensa@ucm.es}

\address{Department of Mathematics, Western University, London, Canada.}
\email{sinnamon@uwo.ca}

\thanks{$^*$ The author was partially supported by grant PID2020-113048GB-I00, funded by MCIN/AEI/ 10.13039/501100011033.}

\thanks{$^{**}$ Support from the Natural Sciences and Engineering Research Council of Canada is gratefully acknowledged.}


\begin{abstract} The distance from the identity operator $I$ to $H^*$, the dual of the Hardy averaging operator, is studied on the cone of nonnegative, nonincreasing functions in Lebesgue space. The exact value is obtained. Optimal lower bounds are also given for difference, $H^*-I$, of these two operators acting on the same cone. A positive answer is given to a conjecture made in ``The norm of Hardy-type oscillation operators in the discrete and continuous settings'' by A. Ben Said, S. Boza, and J. Soria. Preprint, 2024. In addition, a direct comparison, with optimal constants, is given between the operators $H-I$ and $H^*-I$ acting on the cone.
\end{abstract}

\maketitle

\section{Introduction}
The classical Hardy averaging operator $H$ and its dual operator $H^*$ are given by 
$$
 Hf(x)=\frac1x\int_0^xf(t)\,dt\quad\text{and}\quad H^*f(x)=\int_x^\infty f(t)\,\frac{dt}t.
$$
They are bounded operators on $L^p=L^p(0,\infty)$ when $1<p<\infty$, but are not bounded below. Indeed, neither is bounded below even when restricted to nonnegative functions in $L^p$. However, on nonnegative functions, the norms $\|Hf\|_p$ and $\|H^*f\|_p$ are equivalent to each other. Kolyada, in \cite{Ko2}, completed the investigation begun in \cite{Krug} and \cite{BoSo} on the optimal constants for this equivalence. A key idea in this work connects it with finding optimal upper and lower bounds for the operator $H-I$ restricted to nonnegative, nonincreasing functions. 
\begin{theorem}[\cite{Ko2}]
\label{Kol2}
Let $1<p<\infty$ and let $f$ be a nonnegative, nonincreasing function. Then
\begin{align}
\label{K11}
(p-1)^{-1/p}\|f\|_p &\leq \|(H-I)f\|_p \leq (p-1)^{-1}\|f\|_p,\\
\intertext{if $1<p\leq 2$, and}
\label{K22}
(p-1)^{-1}\|f\|_p &\leq \|(H-I)f\|_p \leq (p-1)^{-1/p}\|f\|_p,
\end{align}
if $2 \leq p <\infty$. The constants $p-1$ and $(p-1)^{1/p}$ in \eqref{K11} and \eqref{K22} are the best possible.
\end{theorem}
The right-hand inequalities in \eqref{K11} and \eqref{K22} show that the distance from $H$ to $I$, operating on the cone of nonnegative, nonincreasing functions in $L^p$, is $(p-1)^{-1}$ when $1<p\le2$ and is $(p-1)^{-1/p}$ when $2\le p<\infty$.

The operator $H-I$ arises naturally in other, widely separated contexts, for example, the complex Beurling-Ahlfors transform reduces to $H-I$ when restricted to radial functions, see \cite{Ba} and \cite{St}. When $p=2$, $H-I$ is a Hilbert space isometry that is unitarily equivalent to the unilateral shift, see \cite{BHS}. It has been linked to the Laguerre polynomials in \cite{Krug}. In the theory of interpolation of operators, it acts on rearrangements of functions (nonnegative, nonincreasing functions) to give an equivalent norm in Lorentz spaces, see \cite[page 384]{BS}.

Recent work on $H-I$ and its discrete analogue may be found in \cite{BJ,Si,J,St2,BJ2} but systematic investigation of its dual operator $H^*-I$ is only beginning. The authors of \cite{ABS} proved that the distance from $H^*$ to $I$, operating on the cone of nonnegative, nonincreasing functions in $L^p$, is $p-1$ when $p\ge 2$ and is $2/e$ when $p=1$. They also conjectured that when $1<p<2$, the distance is $C_p^{1/p}$, where 
\begin{equation}\label{const}
C_p=\int_0^1\Big|(H^*-I)\chi_{(0,1)}(x)\Big|^p\,dx=\int_0^1|1+\ln x|^p\,dx.
\end{equation}
That is, they conjectured that for all nonnegative, nonincreasing functions $f$,
$$
\|(H^*-I)f\|_p \le C_p^{1/p}\|f\|_p
$$
and $C_p^{1/p}$ is the smallest constant for which this holds.
 
In the present article, this conjecture is verified and optimal lower bounds for $H^*-I$ are established. For easy comparison, results are presented in the form of Theorem \ref{Kol2}. Our main result  is the following theorem.
\begin{theorem}\label{main}
Let $1\le p<\infty$. If $f$ is nonnegative and nonincreasing, then
\begin{align}
\label{U1}
(p-1)\|f\|_p&\leq \|(H^{*}-I)f\|_p \leq C_p^{1/p}\|f\|_p\\
\intertext{if $1\le p\leq 2$ and}
\label{U2}
C_p^{1/p}\|f\|_p&\leq \|(H^{*}-I)f\|_p \leq (p-1)\|f\|_p
\end{align}
if $2 \leq p <\infty$. The constants $p-1$ and $C_p^{1/p}$ are optimal in both \eqref{U1} and \eqref{U2}.
\end{theorem}

The case $p=1$ of \eqref{U1} and the second inequality in \eqref{U2} are included here for the convenience of the reader. Both were proved in \cite{ABS}, including optimality of the constants $C_1=2/e$ and $p-1$. We do not reproduce the proofs here. The remaining three inequalities, in the case $p>1$, are proved in Theorems \ref{T} and \ref{cp} below.

Some consequences of the main results are given in Corollary \ref{sqrd}, which shows that the results of Theorem \ref{main} remain optimal on a smaller cone, and in Corollary \ref{H-I}, which gives optimal constants for inequalities relating $H-I$ and $H^*-I$ on the original cone of nonnegative, nonincreasing functions.

The final section of the paper is concerned with properties of the constant $C_p$ considered as a function of $p$.

\begin{remark} If $f$ is a nonnegative, measurable function on $(0,\infty)$ that takes the value $\infty$ on a set of measure zero, then $f$, $Hf$, $H^*f$, $(H-I)f$ and $(H^*-I)f$ are well-defined but may take the value $\infty$ on a set of positive measure. However, if $f$ is allowed to take the value $\infty$ on a set of positive measure, $(H-I)f$ and $(H^*-I)f$ may not be well-defined because $\infty-\infty$ may occur. This is why we must restrict the domain of the operator $(H^*-I)H^*$ encountered in Corollary \ref{sqrd}.
\end{remark}

\section{Main results}
We begin by introducing a family of test functions that will be used in Theorem \ref T and Corollaries \ref{sqrd} and \ref{H-I}.
\begin{lemma}\label{S} Let $p>1$, $g_q(x)=\frac1qx^{-1/q}\chi_{(1,\infty)}(x)$ and $f_q=H^*g_q$ for $1<q<p$. Then $g_q,f_q,f_q-g_q\in L^p$ are nonnegative, $f_q, f_q-g_q$ are nonincreasing,  
\begin{align}
\lim_{q\to p^-}\frac{\|(H^*-I)f_q\|_p}{\|f_q\|_p}=\lim_{q\to p^-}\frac{\|({H^*}^2-H^*)g_q\|_p}{\|H^*g\|_p}
&=p-1,\label{test1}\\
\intertext{and}
\lim_{q\to p^-}\frac{\|(H^*-I)(f_q-g_q)\|_p}{\|(H-I)(f_q-g_q)\|_p}=\lim_{q\to p^-}\frac{\|(H^*-I)^2g_q\|_p}{\|g_q\|_p}&=(p-1)^2\label{test2}.
\end{align}
\end{lemma}
\begin{proof} Observe that $0\le g_q\in L^p$ and set $\varepsilon=q^{(p-1)/p}(p-q)^{1/p}$ so that $\varepsilon\|g_q\|_p=1$. 

Fix a polynomial $U$. Let $V$ be the polynomial satisfying $V(z)z=U(z+p)-U(p)$. An easy calculation shows that $(H^*-qI)g_q=\chi_{(0,1)}$ and it follows that
$$
U(H^*)g_q-U(pI)g_q=V(H^*-pI)(H^*-pI)g_q=V(H^*-pI)(\chi_{(0,1)}+(q-p)g_q).
$$
Since $H^*$ is bounded on $L^p$, so is $V(H^*-pI)$. Let $N$ denote its operator norm. By the triangle inequality and the calculation $\|\chi_{(0,1)}\|_p=1$,
$$
\big|\|U(H^*)g_q\|_p-\|U(p)g_q\|\big| \le\|V(H^*-pI)(\chi_{(0,1)}+(q-p)g_q)\|_p
\le N(1+(p-q)\|g_q\|_p).
$$
Multiplying through by $\varepsilon$ and using $\varepsilon\|g_q\|_p=1$ yields
$$
\big|\varepsilon\|U(H^*)g_q\|_p-|U(p)|\big|
\le \varepsilon N+N(p-q)\to0
$$
as $q\to p^-$. (Clearly, $\varepsilon\to0$ as $q\to p^-$.) We conclude that 
\begin{equation}\label{poly}
\lim_{q\to p^-}\varepsilon\|U(H^*)g_q\|=|U(p)|.
\end{equation}
It is easy to see that $f_q=\chi_{(0,1)}+qg_q$ and $f_q-g_q=\chi_{(0,1)}+(q-1)g_q$ are nonnegative and nonincreasing. The definition of $f_q$ and the identity $(H-I)(H^*-I)=I$ prove the first equations in both \eqref{test1} and \eqref{test2}. Choose appropriate polynomials $U$
 in \eqref{poly} to verify the second equations in \eqref{test1} and \eqref{test2}. This completes the proof.
\end{proof}

A known result for $H-I$ is the key to proving the first inequality in \eqref{U1}.
\begin{theorem}
\label{T}
Let $1<p\leq 2$ and let $f$ be a nonnegative, nonincreasing function. Then
 \begin{equation}
\label{D1}
 (p-1)\|f\|_p \leq \|(H^{*}-I)f\|_p.
\end{equation}
The constant $p-1$ in \eqref{D1} is optimal.
\end{theorem}
\begin{proof} The inequality
\begin{equation}
\label{A35}
\|(H-I)g\|_p \leq (p-1)^{-1}\|g\|_p,\quad g\in L^p,
\end{equation}
appears as \cite[Theorem 4.1]{Ba}. For other proofs of \eqref{A35}, see \cite{St} and those cited in its introduction. An elementary proof appears in \cite[Corollary 5]{Si}. 

If $(H^*-I)f\notin L^p$, inequality \eqref{D1} holds trivially. If $(H^*-I)f\in L^p$, then inequality \eqref{A35}, with $g=(H^*-I)f$, and the identity $(H-I)(H^*-I)=I$ show that $f\in L^p$ and
\begin{equation*}
(p-1)\|f\|_p = (p-1)\|(H-I)(H^*-I)f\|_p \leq \|(H^*-I)f\|_p
\end{equation*}
as required.

Equation \eqref{test1} of Lemma \ref{S} shows that if $p-1$ were replaced by a larger constant, then inequality \eqref{D1} would fail for some nonnegative, nonincreasing function $f_q$. So $p-1$ is optimal in \eqref{D1}.
\end{proof}

The second inequality in \eqref{U1} and the first inequality in \eqref{U2} will be handled together. The next lemma isolates a technical portion of the proof of Theorem \ref{cp}. 

We use the differentiation formula 
$\tfrac d{dz}|z|^q=q|z|^{q-2}z$ in the proof. It is valid for all nonzero $z$ and all real $q$. Note that if $q>1$, the right-hand side extends continuously to $z=0$.

\begin{lemma}\label{tech} Let $u>0$ and $0\le t\le 1$. For $r\ge0$, set 
$$
h(r)=e^r\int_0^{e^{-r}}|1+\ln x|^p\,dx.
$$
If $1<p\le 2$, then 
$$
h(t^{1/p}u)\le (1-t)h(0)+th(u).
$$
If $p>2$, then
$$
h(t^{1/p}u)\ge (1-t)h(0)+th(u).
$$
\end{lemma}
\begin{proof} To cover both cases, we write ``$\le^*$'' to mean ``$\le$'' when $1<p\le 2$ and to mean ``$\ge$'' when $p\ge2$. 

We will use $h(r)$ in a different form. Substitute $x=e^{-y}$ to obtain 
$$
h(r)=e^r\int_r^\infty|y-1|^pe^{-y}\,dy.
$$

Now fix $u>0$ and let $r(t)=t^{1/p}u$. We will write $r$, $r'$, and $r''$ instead of $r(t)$, $r'(t)$, and $r''(t)$. Observe that $r''r=(1-p)(r')^2$.

Since 
$$
\frac d{dy}(|y-1|^pe^{-y})=p|y-1|^{p-2}(y-1)e^{-y}-|y-1|^pe^{-y}
$$
is continuous, we may integrate both sides and simplify to get
\begin{equation}
\label{Id1}
pe^r\int_r^\infty|y-1|^{p-2}(y-1)e^{-y}\,dy=h(r)-|r-1|^p.
\end{equation}
Also, since $|y-1|^{p-2}(y-1)e^{-y}$ extends to be continuous at $y=1$ and its derivative,
\begin{align*}
\frac d{dy}&(|y-1|^{p-2}(y-1)e^{-y})\\
&=(p-2)|y-1|^{p-4}(y-1)^2e^{-y}+|y-1|^{p-2}e^{-y}-|y-1|^{p-2}(y-1)e^{-y}\\
&=(p-1)|y-1|^{p-2}e^{-y}-|y-1|^{p-2}(y-1)e^{-y},
\end{align*}
is continuous except at $y=1$ we may integrate both sides and simplify to get
\begin{equation}
\label{Id2}
(p-1)e^r\int_r^\infty|y-1|^{p-2}e^{-y}\,dy
=e^r\int_r^\infty|y-1|^{p-2}(y-1)e^{-y}\,dy-|r-1|^{p-2}(r-1).
\end{equation}
Next, we compute two derivatives of $h(r)$ with respect to $t$. By \eqref{Id1},
$$
\frac d{dt}h(r)=(h(r)-|r-1|^p)r'=\Big(pe^r\int_r^\infty|y-1|^{p-2}(y-1)e^{-y}\,dy\Big)r'.
$$
By \eqref{Id2}, and the observation $r''r=(1-p)(r')^2$, 
\begin{align*}
\frac {d^2}{dt^2}h(r)&=p\Big(e^r\int_r^\infty|y-1|^{p-2}(y-1)e^{-y}\,dy-|r-1|^{p-2}(r-1)\Big)(r')^2\\
&+\Big(pe^r\int_r^\infty|y-1|^{p-2}(y-1)e^{-y}\,dy\Big)r''\\
&=\frac pr(p-1)e^r(r')^2\Big(r\int_r^\infty|y-1|^{p-2}e^{-y}\,dy-\int_r^\infty|y-1|^{p-2}(y-1)e^{-y}\,dy\Big).
\end{align*}

Now let
$$
g(s)=s\int_s^\infty|y-1|^{p-2}e^{-y}\,dy-\int_s^\infty|y-1|^{p-2}(y-1)e^{-y}\,dy.
$$
Taking $r=0$ in \eqref{Id1} shows
$$
g(0)=-\int_0^\infty|y-1|^{p-1}(y-1)e^{-y}\,dy=\tfrac1p(1-h(0)). 
$$
But H\"older's inequality and a bit of calculus shows that 
$$
h(0)^{1/p}=\Big(\int_0^1|1+\ln x|^p\,dx\Big)^{1/p}\le^*\!\Big(\int_0^1|1+\ln x|^2\,dx\Big)^{1/2}=1,
$$
so $0\le^*\!g(0)$. In particular, if $p=2$, then $g(0)=0$.

Clearly, 
$$
\lim_{s\to\infty}\int_s^\infty|y-1|^{p-2}(y-1)e^{-y}\,dy=0.
$$
Also, by l'Hospital's rule
$$
\lim_{s\to\infty}\frac{\int_s^\infty|y-1|^{p-2}e^{-y}\,dy}{s^{-1}}
=\lim_{s\to\infty}\frac{-|s-1|^{p-2}e^{-s}}{-s^{-2}}=0.
$$
Thus, $g(s)\to0$ as $s\to\infty$. 

If $s\ne1$, 
$$
g'(s)=\int_s^\infty|y-1|^{p-2}e^{-y}\,dy-|s-1|^{p-2}e^{-s}\ \ \text{and}\ \ 
g''(s)=(p-2)(1-s)|s-1|^{p-4}e^{-s}.
$$

If $p=2$, $g'(s)=0$ and $g(s)=g(0)=0$. If $1<p<2$, $g$ is concave on $(0,1)$ and convex on $(1,\infty)$. It follows that $g$ decreases to zero on $(1,\infty)$, that $g(1)>0$, that $g$ has no local minimum in $(0,1)$ and that $g\ge0$ on $(0,1)$. If $p>2$, $g$ is convex on $(0,1)$ and concave on $(1,\infty)$. It follows that $g$ increases to zero on $(1,\infty)$, that $g(1)<0$, that $g$ has no local maximum on $(0,1)$ and that $g\le0$ on $(0,1)$. We have shown that $0\le^*\!g(s)$ for $0<s<\infty$ and, taking $s=r$, that $0\le^*\frac {d^2}{dt^2}h(r)$. 

We conclude that $h(r)$ is a convex function of $t$ on $[0,1]$ when $1<p\le2$ and $h(r)$ is a concave function of $t$ on $[0,1]$ when $p\ge2$. When $t=0$, $r=0$ and when $t=1$, $r=u$. Thus, for any $t\in [0,1]$ we have $h(t^{1/p}u)=h(r)\le^*\!(1-t)h(0)+th(u)$ as required.
\end{proof}

\begin{theorem}\label{cp} Let $f$ be a nonnegative, nonincreasing function. If $1<p\le 2$, then
$$
\|(H^*-I)f\|_p^p\le C_p\|f\|_p^p.
$$
If $p\ge2$, then
$$
\|(H^*-I)f\|_p^p\ge C_p\|f\|_p^p.
$$
In both cases, $C_p$ is the optimal constant for which the inequality holds.
\end{theorem}
\begin{proof} To cover both cases, we write ``$\le^*$'' to mean ``$\le$'' when $1<p\le 2$ and to mean ``$\ge$'' when $p\ge2$. Taking $f=\chi_{(0,1)}$ reduces both inequalities above to equality, which will show that the constant $C_p$ is optimal once we prove that both inequalities hold. 
As in Lemma \ref{tech}, let 
$$
h(r)=e^r\int_0^{e^{-r}}|1+\ln x|^p\,dx.
$$
Observe that $h(0)=C_p$.

As mentioned earlier, the operator $H^*-I$ is bounded above and below on $L^p$. Thus, it suffices to prove the theorem for simple functions $f$ of the following form: Let $a_0=0<a_1<\dots<a_N<\infty$, $b_1>\dots>b_N>0=b_{N+1}$ and set $d_n=\ln a_n$ for $n=0,\dots,N$, so $d_0=-\infty<d_1<\dots<d_N$. Define
$$
f=\sum_{n=1}^N b_n\chi_{(a_{n-1},a_n]}.
$$
For convenience, set
$$
A_n=\int_{a_{n-1}}^{a_n}\bigg|\int_x^\infty f(t)\,\frac{dt}t-f(x)\bigg|^p\,dx, \quad n=1,2,\dots,N. 
$$
Since $(H^*-I)f$ is supported on $(0,a_N]$,
$$
\|(H^*-I)f\|_p^p=\sum_{n=1}^N A_n.
$$
If $a_{n-1}<x\le a_n$, then 
\begin{align*}
\int_x^\infty f(t)\,\frac{dt}t-f(x)
&=\int_x^{a_n}b_n\,\frac{dt}t
+\int_{a_n}^{a_{n+1}}b_{n+1}\,\frac{dt}t+\dots+\int_{a_{N-1}}^{a_N}b_N\,\frac{dt}t-b_n\\
&=b_n(d_n-\ln x)+b_{n+1}(d_{n+1}-d_n)+\dots+b_N(d_N-d_{N-1})-b_n\\
&=b_n(S_n-1-\ln x),
\end{align*}
where $S_n$ is chosen so that 
$$
b_nS_n=b_nd_n+\sum_{k=n+1}^N b_k(d_k-d_{k-1}).
$$
Observe that $b_n(S_n-d_n)=b_{n+1}(S_{n+1}-d_n)$ for $n=1,\dots, N-1$ and $b_N(S_N-d_N)=0$. Let $\ln z=\ln x-S_n$ to get
\begin{align*}
\frac{A_n}{b_n^p}=\int_{a_{n-1}}^{a_n} |S_n-1-\ln x|^p\,dx&=e^{S_n}\int_{e^{d_{n-1}-S_n}}^{e^{d_n-S_n}}|1+\ln z|^p\,dz\\
&=a_nh(S_n-d_n)-a_{n-1}h(S_n-d_{n-1})\\
&=a_nh((b_{n+1}/b_n)(S_{n+1}-d_n))-a_{n-1}h(S_n-d_{n-1}).
\end{align*}
Taking $t=(b_{n+1}/b_n)^p$ and $u=S_{n+1}-d_n$ in Lemma \ref{tech}, we have
$$
a_nh((b_{n+1}/b_n)(S_{n+1}-d_n))
\le^*\! a_n(1-(b_{n+1}/b_n)^p)C_p+a_n(b_{n+1}/b_n)^ph(S_{n+1}-d_n).
$$
Multiplying through by $b_n^p$, we get
$$
A_n\le^*\! a_n(b_n^p-b_{n+1}^p)C_p+[b_{n+1}^pa_nh(S_{n+1}-d_n)-b_n^pa_{n-1}h(S_n-d_{n-1})].
$$
When we sum over $n$, the bracketed terms telescope to zero, so
$$
\sum_{n=1}^NA_n\le^*\! C_p\sum_{n=1}^Na_n(b_n^p-b_{n+1}^p)=C_p\sum_{n=1}^Nb_n^p(a_n-a_{n-1})=C_p\int_0^\infty f(t)^p\,dt=C_p\|f\|_p^p.
$$
This completes the proof.
\end{proof}

Let $g$ be a nonnegative, measurable function on $(0,\infty)$. Then
$$
{H^*}^2g(x)=H^*(H^*g)(x)= \int_x^\infty \ln (s/x)g(s)\,\frac{ds}s.
$$
The following result shows that if $1 < p < \infty$, then for all such $g\ge0$ for which $({H^*}^2-H^*)g$ is defined, the $L^p$-norms of $({H^*}^2-H^*)g$ and $H^*g$ are equivalent. 
\begin{corollary}\label{sqrd}
Let $1<p<\infty$ and let $g$ be a nonnegative, measurable function such that $H^*g(x)<\infty$ for all $x>0$. Then,
\begin{align}
\label{A1}
(p-1)\|H^*g\|_p &\leq\|({H^*}^2-H^*)g\|_p\leq C_p^{1/p}\|H^*g\|_p\\
\intertext{if $1<p\leq 2$, and}
\label{A2}
C_p^{1/p}\|H^*g\|_p&\leq\|({H^*}^2-H^*)g\|_p\leq (p-1)\|H^*g\|_p
\end{align}
if $2 \leq p <\infty$. The constants $p-1$ and $C_p^{1/p}$ are optimal in both \eqref{A1} and \eqref{A2}.
\end{corollary}
\begin{proof}
Set $f=H^*g$ and observe that $f$ is nonnegative and nonincreasing. Since $({H^*}^2-H^*)g=(H^*-I)f$, Theorem \ref{main} shows that inequalities \eqref{A1} and \eqref{A2} both hold.

Equation \eqref{test1} of Lemma \ref{S} shows that if $p-1$ were replaced by larger constant in inequality \eqref{A1} or by a smaller constant in inequality \eqref{A2}, then that inequality would fail for some nonnegative function $g_q$ for which $H^*g_q(x)<\infty$ for all $x>0$. This shows $p-1$ is optimal in both \eqref{A1} and \eqref{A2}.

It remains to show that the constant $C_p^{1/p}$ is optimal in both \eqref{A1} and \eqref{A2}. Let $k_\varepsilon(s)=(s/\varepsilon)\chi_{(1-\varepsilon,1)}(s)$. Then $H^*k_\varepsilon(t)=\min(1,(1-t)/\varepsilon)\chi_{(0,1)}(t)$. A simple sketch shows that as $\varepsilon$ goes to zero, the piecewise linear functions $H^*k_\varepsilon$ increase pointwise to $\chi_{(0,1)}$. The monotone convergence theorem implies that $H^*k_\varepsilon$ converges to  $\chi_{(0,1)}$ in $L^p$. But $H^*$ and $H^*-I$ are continuous maps on $L^p$, so we have
$$
\|H^*k_\varepsilon\|_p\to\|\chi_{(0,1)}\|_p=1\quad\text{and}\quad\|({H^*}^2-H^*)k_\varepsilon\|_p\to\|(H^*-I)\chi_{(0,1)}\|_p=C_p^{1/p}.
$$
It follows that $C_p^{1/p}$ is the smallest constant for which \eqref{A1} holds and the largest constant for which \eqref{A2} holds. This completes the proof.
\end{proof}
Combining Theorems \ref{main} with \ref{Kol2} gives the best constants in inequalities that relate $H^*-I$ and $H-I$ for decreasing functions.
\begin{corollary}\label{H-I} Let $1<p<\infty$ and let $f$ be a nonnegative, nonincreasing function. Then
 \begin{equation}
\label{E1}
(p-1)^2\|(H-I)f\|_p\leq \|(H^*-I)f\|_p \leq C_p^{1/p}(p-1)^{1/p}\|(H-I)f\|_p,
\end{equation}
if $1<p\leq 2$, and
\begin{equation}
\label{E2}
C_p^{1/p}(p-1)^{1/p}\|(H-I)f\|_p\leq \|(H^*-I)f\|_p \leq (p-1)^2\|(H-I)f\|_p,
\end{equation}
if $2 \leq p <\infty$. The constants $C_p^{1/p}(p-1)^{1/p}$ and $(p-1)^2$ are the best possible in both \eqref{E1} and \eqref{E2}.
\end{corollary}
\begin{proof}
First suppose $1\le p\le2$. Using \eqref{U1} of Theorem \ref{main} and \eqref{K11} of Theorem \ref{Kol2} we get
$$
(p-1)^2\|(H-I)f\|_p\le (p-1)\|f\|_p\le \|(H^*-I)f\|_p
$$
and 
$$
\|(H^*-I)f\|_p\le C_p^{1/p}\|f\|_p\le C_p^{1/p}(p-1)^{1/p}\|(H-I)f\|_p
$$
to give \eqref{E1}.
Next we suppose $2\le p<\infty$. Using \eqref{U2} of Theorem \ref{main} and \eqref{K22} of Theorem \ref{Kol2} we get
$$
C_p^{1/p}(p-1)^{1/p}\|(H-I)f\|_p\le C_p^{1/p}\|f\|_p\le \|(H^*-I)f\|_p
$$
and 
$$
\|(H^*-I)f\|_p\le (p-1)\|f\|_p\le (p-1)^2\|(H-I)f\|_p
$$
to give \eqref{E2}.

To see that the constant $C_p^{1/p}(p-1)^{1/p}$ is optimal, take $f=\chi_{(0,1)}$. We get $\|(H^*-I)f\|_p=C_p^{1/p}$ from \eqref{const}. Also, 
$$
(H-I)f(x)=\frac1x\int_0^x\chi_{(0,1)}(t)\,dt-\chi_{(0,1)}(x)
=\frac1x\chi_{(1,\infty)}(x)
$$
so 
$$
\|(H-I)f\|_p=\Big(\int_x^\infty x^{-p}\,dx\Big)^{1/p}=(p-1)^{-1/p}.
$$
It follows that no smaller constant is possible in the second inequality of \eqref{E1} and no larger constant is possible in the first inequality of \eqref{E2}.

For optimality of the constant $(p-1)^2$, we apply Lemma \ref{S}, taking $f$ to be $f_q-g_q$. Equation \eqref{test2} shows that no larger constant will satisfy the first inequality in \eqref{E1} and no smaller constant will satisfy the second inequality in \eqref{E2}.
\end{proof}
\section{Behavior of the optimal constants}
Recall that 
$$
C_p=\int_0^1|1+\ln x|^p\,dx.
$$
Both $C_p$ and $C_p^{1/p}$ appear above as best constants. Here they are considered as functions of $p$. We list some easily obtained properties of $C_p$ and $C_p^{1/p}$.

\begin{enumerate}[label = (\alph*), leftmargin=2em]
\item The integral defining $C_p$ converges for $0<p<\infty$.

\item A bit of calculus shows that $C_1=2/e$ and $C_2=1$. 

\item $C_p^{1/p}$ is a strictly increasing function of $p$ for $0<p<\infty$: If $0<p<q<\infty$, then H\"older's inequality with indices $q/p$ and $q/(q-p)$ implies
$$
C_p^{1/p}=\Big(\int_0^1|1+\ln x|^p\,dx\Big)^{1/p}
\le\Big(\int_0^1|1+\ln x|^q\,dx\Big)^{1/q}\Big(\int_0^1\,dx\Big)^{(q-p)/(pq)}=C_q^{1/q}.
$$
Since $|1+\ln x|$ is not constant, the inequality is strict. 

\item\label{LL} Well-known H\"older's inequality arguments, much like the one above, show that $\ln(C_p)$ and $\ln(C_{1/p}^p)$ are convex functions of $p$ on $(0,\infty)$.

\item $C_p$ is a strictly increasing function of $p$ for $2\le p<\infty$: If $2\le p<q$, then $C_q^{1/q}>C_2^{1/2}=1$ so $C_q>1$ and we have $C_q<C_q^{q/p}$. This shows
$$
C_p< C_q^{p/q}\le(C_q^{q/p})^{p/q}=C_q.
$$

\item $C_p$ is not an increasing function of $p$ for $0<p\le 1$: If $0<p\le1$, the inequality 
$$
|1+\ln x|^p\le (|1+\ln x|+1)^p\le |1+\ln x|+1,
$$
together with the dominated convergence theorem, shows that $\lim_{p\to 0^+} C_p=1$. But $C_1=2/e<1$.

\item $C_p^{1/p}$ does not tend to zero as $p\to0$: By l'Hospital's rule and a dominated convergence argument similar to the one in the previous item, 
$$
\lim_{p\to 0^+}\frac{\ln(C_p)}p=\lim_{p\to 0^+}\frac{\int_0^1|1+\ln x|^p\ln(|1+\ln x|)\,dx}{\int_0^1|1+\ln x|^p\,dx}=\int_0^1\ln(|1+\ln x|)\,dx.
$$
Thus, 
$$
\lim_{p\to 0^+} C_p^{1/p}=\exp\Big(\int_0^1\ln(|1+\ln x|)\,dx\Big).
$$
\item $C_p$ is a strictly increasing function of $p$ for $1\le p\le 2$. This is included in the next proposition.
\end{enumerate}
 
\begin{proposition}
\label{constante} For $p\ge1$, $C_p$ is a strictly increasing function of $p$. 
\end{proposition}
\begin{proof} Let $p>0$. For each positive base $b\ne1$, the exponential function $p\mapsto b^p$ is strictly convex. It follows that if $0<p_0<p_1$, $0<\theta<1$, and $p=(1-\theta)p_0+\theta p_1$, then $b^p<(1-\theta)b^{p_0}+\theta b^{p_1}$. Taking $b=|1+\ln x|$ and integrating with respect to $x$, we get $C_p<(1-\theta)C_{p_0}+\theta C_{p_1}$. That is, $C_p$ is a strictly convex function for $0< p<\infty$. (Because logarithmic convexity implies convexity, this follows from item \ref{LL}, above, but we prefer the direct argument.) In particular, $\frac d{dp}C_p$ is strictly increasing for $0< p<\infty$. Therefore, to show that $C_p$ is strictly increasing for $1\le p<\infty$, it suffices to show that $\frac d{dp}C_p\ge0$ when $p=1$.

For this, we split the integral defining $C_p$ into two parts, letting $y=-1-\ln x$ in the first and letting $y=1+\ln x$ in the second. So
$$
C_p=\int_0^{1/e}(-1-\ln x)^p\,dx+\int_{1/e}^1(1+\ln x)^p\,dx
=\int_0^\infty y^pe^{-y}\,dy+\int_0^1 y^p e^y\,dy.
$$
Recognizing the first term as a gamma function and expressing the second as a power series, we get
$$
C_p=\Gamma(p+1)+\sum_{k=0}^\infty\frac1{k!}\int_0^1 y^{k+p}\,dy
=\Gamma(p+1)+\sum_{k=0}^\infty\frac1{k!}\frac1{k+p+1}.
$$
Differentiating the uniformly convergent sum term by term yields
$$
\frac d{dp}C_p=\Gamma'(p+1)-\sum_{k=0}^\infty\frac1{k!}\frac1{(k+p+1)^2}.
$$

Taking $p=1$ in the first term we use the formula $\Gamma'(2)=1-\gamma$. Here $\gamma$ is the Euler-Mascheroni constant, which is known to have approximate value $0.5772156649\dots$. Thus $\Gamma'(2)>1-0.578=0.421$.

Taking $p=1$ in the second term, we have
$$
\sum_{k=0}^\infty\frac1{k!}\frac1{(k+2)^2}<\frac14+\frac19+\sum_{k=2}^\infty\frac{1}{(k+2)!}
= \frac14+\frac19+e-\Big(1+1+\frac12+\frac16\Big).
$$
A calculation shows this is less than $0.42$ and we obtain $\frac d{dp}C_p>0.421-0.42>0$ when $p=1$. This completes the proof.
\end{proof}

\end{document}